\documentclass[12pt]{amsart}
\usepackage{latexsym,amsmath,amssymb,esint,enumitem}
\usepackage{mathtools,slashed}
\usepackage{graphicx,subfigure,psfrag}
\usepackage{hyperref}
\usepackage{xcolor}
\hypersetup{
    colorlinks,
    linkcolor={red!50!black},
    citecolor={blue!50!black},
    urlcolor={blue!80!black}
}

\usepackage{tikz}
\usetikzlibrary{positioning}
\usetikzlibrary {arrows.meta}

\usepackage{graphicx}

\newcommand{\eps}{\varepsilon}
\newcommand{\HH}{{\mathbb H}}
\newcommand{\PP}{{\mathbb P}}
\newcommand{\RR}{{\mathbb R}}
\renewcommand{\SS}{{\mathbb S}}

\newcommand{\Vol}{\mathrm{vol}\,}
\newcommand{\ZZ}{{\mathbb Z}}

\newtheorem{theorem}{Theorem}
\newtheorem*{theorem*}{Theorem}
\newtheorem{lemma}[theorem]{Lemma}
\newtheorem{corr}[theorem]{Corollary}
\newtheorem{prop}[theorem]{Proposition}

\theoremstyle{definition}

\newtheorem*{deff*}{Definition}

\numberwithin{theorem}{section}
\numberwithin{equation}{section}

\begin{document}

\title[ Monotonicity of heat kernels]
{On monotonicity of heat kernels:\\ a new example and counterexamples}

\author[A. Burchard]{Almut Burchard} \address{University of Toronto, Canada} 
\email{almut@math.toronto.edu}

\author[A. D. Mart\'inez]{\'Angel D. Mart\'inez} \address{CUNEF Universidad, Spain} \email{angeld.martinez@cunef.edu}
%\date{Revised August 31, 2024.\\

\begin{abstract} 
 We discover a new, non-radial example of a manifold whose heat 
kernel decreases monotonically along all minimal geodesics. 
We also  classify the flat tori with this monotonicity property. 
Furthermore, we show that for a generic metric on any smooth manifold 
 the monotonicity property fails at large times.
 This answers a recent question of
Alonso-Or\'an, Chamizo, Mart\'inez, and Mas.

\end{abstract}

\maketitle

%%%%%%%%%%%%%%%%%%%%%%%%%%%%%%%%%%%%%%%%%%%%%%%%%%%%%%%%%%%%%%%%%%%%%%
\section{\bf Introduction}\label{sec-intro}
%%%%%%%%%%%%%%%%%%%%%%%%%%%%%%%%%%%%%%%%%%%%%%%%%%%%%%%%%%%%%%%%%%%%%%

The heat kernel 
on a manifold contains a wealth of local 
and global geometric information about the underlying space. 
 It is of central importance
for partial differential equations (describing
diffusion of a unit of heat released from a point and diffusing
through the manifold) and for probability 
(giving the transition densities for
Brownian motion).  Although the asymptotics of the heat kernel are well-understood both in the large or small time regimes, less is known
about its geometric properties at fixed, positive times.
 The main purpose of this paper is to study monotonicity of this kernel along geodesics.

In order to fix ideas, let us recall that on 
Euclidean space $\RR^n$ the heat kernel is given by 
\[
K_t(x,y)=(4\pi t)^{-n/2}\exp\bigl(-\tfrac{|x-y|^2}{4t}\bigr)\,.
\]
The fact that this kernel
is a strictly decreasing function of
the distance 
$|x-y|$ implies many classical inequalities
in Functional Analysis and Probability, such as
the isoperimetric inequality,
the P\'olya-Szeg\H{o} inequality for the 
Dirichlet energy ~\cite[Lemma 7.17]{LL}, 
the Faber-Krahn inequality for the fundamental
frequency of a domain,
as well as 
inequalities for Brownian motion and
path integrals~\cite{Lutt} (cf.
\cite[Chapter 3]{LL} or \cite[Chapter 8]{B}).

By definition,
the heat kernel on a Riemannian manifold is the fundamental solution of 
the heat equation $\partial_tu=\Delta u$, where $\Delta$ is the Laplace-Beltrami operator. Explicit expressions for the heat kernel are available only in very few cases.
On the sphere $\SS^n$ there are
no formulas when $n\ge 2$, while for the hyperbolic space
$\HH^n$ there are exact formulas that become increasingly 
cumbersome in higher dimensions.
Nevertheless, the heat kernels on these manifolds share 
the symmetry and monotonicity properties of the Euclidean one:

\begin{theorem}[Radially decreasing heat
kernels~\cite{CY, Ch, Andersson, NSS, ACMM}]
\label{thm-decreasing}
Let $M$ be the Euclidean space $\RR^n$, the sphere $\SS^n$,
or the hyperbolic space $\HH^n$, with the standard
uniform metric. 
For every time $t>0$, the heat kernel $K_t(x,y)$ is a
strictly decreasing function of the geodesic distance $d(x,y)$.
\end{theorem}

A number of different proofs of this result can be found in
the literature, going back to the 
work of Cheeger and Yau from the 1980s. It is a natural 
question to ask what can be said about monotonicity properties 
on other manifolds. In a recent paper, 
Alonso-Or\'an, Chamizo, Mart\'inez, and Mas 
obtain
an improvement in the case of rotationally symmetric manifolds about some point $x$, i.e. manifolds that admit global spherical coordinates $(\rho,\sigma)\in (0,D)\times\mathbb{S}^{n-1}$ such that the metric has the form $d\rho^2+A(\rho)d\sigma^2$ for some function $A$. Here, the diameter $D$ may 
be infinite or not. 
If the diameter is finite,
their result implies the following.

\begin{theorem}[Symmetric decrease
about a point~\cite{ACMM}]
\label{thm-decreasing-point} Let $(M,g)$ be
a compact smooth Riemannian manifold that is 
rotationally symmetric about  some point $x\in M$. 
For every time $t>0$
the heat kernel $K_t(x,y)$ 
is a strictly decreasing function of the 
geodesic distance $d(x,y)$. 
\end{theorem}

In particular,
the conclusion of Theorem~\ref{thm-decreasing} extends to 
the real projective space $\RR\PP^n$, endowed with the 
uniform metric. The proof given in~\cite{ACMM} 
applies
the parabolic maximum principle to radial solutions of the 
heat equation, expressed in normal coordinates about $x$. 
This also yields a new proof of Theorem~\ref{thm-decreasing}. 
We refer the interested reader 
to \cite{ACMM} for a more detailed
discussion and a complete bibliography on the topic.

As observed in~\cite{ACMM}, in the absence
of symmetry $K_t$ cannot be a function of the 
geodesic distance alone.
On a general Riemannian manifold, the authors
propose to consider instead
the monotonicity of the heat kernel along geodesics,
a much weaker property.
They raise the question (we paraphrase freely)
{\em whether the heat kernel $K_t(x,\gamma(s))$ decreases 
along geodesics $\gamma$ emanating 
from $x$? Can such a monotonicity  property
hold for all times $t>0$, up to the cut locus of $x$?} 
They observe that flat rectangular tori,   and,
more generally, any product of manifolds
with geodesically decreasing heat kernels inherits this
property from their factors.

In this paper, we describe a new example 
of geodesic monotonicity, given by 
the flat torus associated with
a regular triangular (`honeycomb') lattice in the plane.
This torus appears as 
an extremal for numerous spectral 
and isoperimetric inequalities 
(see, for instance~\cite{B-1997, EI, H, OMC, FS} for
some recent results).
It is distinguished by a
six-fold reflection symmetry, which manifests 
itself in hexagonal fundamental domain of the lattice.
We use the maximum principle to show that
the heat kernel $K_t(x,y)$
decreases as $y$ moves along any
geodesic emanating from $x$ until
it meets the boundary of the fundamental hexagon centered at~$x$.

However, in general the answer to
the above question is negative. We present 
several results that indicate just how rare
it is for a Riemannian metric to have
a geodesically decreasing heat kernel.
Among flat two-dimensional tori, the only 
other examples are those associated with 
rectangular lattices (see Theorem~\ref{thm-flat}). The heat kernel on a flat Klein bottle cannot be geodesically decreasing for large $t$.
In an upcoming paper, we prove that on $\SS^2$, the only Riemannian metrics with geodesically decreasing heat kernels~\cite{BM-S2} are the uniform metrics. On any smooth Riemannian manifold, geodesic
monotonicity places narrow constraints on the 
multiplicity of eigenvalues
of the Laplace-Beltrami operator, and on the corresponding
spectral projections (Theorem~\ref{thm-constant}).
It is an open question what other
examples may exist on surfaces of higher genus
and in higher dimensions.

%%%%%%%%%%%%%%%%%%%%%%%%%%%%%%%%%%%%%%%%%%%%%%%%%%%%%%%%%%%%%%%%%%%%%%
\section{\bf Main results}\label{sec-main}
%%%%%%%%%%%%%%%%%%%%%%%%%%%%%%%%%%%%%%%%%%%%%%%%%%%%%%%%%%%%%%%%%%%%%%

We need to fix some basic notions. Clearly,
the monotonicity described above  makes sense only for
minimal geodesics, i.e., for geodesics whose length
equals the distance between their endpoints.

\begin{deff*}[Geodesic monotonicity]
\label{deff-mono}
Let $(M,g)$ be a Riemannian manifold.  
We say that a real-valued function $F$ on $M$  
{\em decreases geodesically about a point} $x\in M$, if the 
function $F(\gamma(s))$ is non-increasing in $s$ 
for every minimal geodesic $\gamma:[0,s_0]\to M$ with
$\gamma(0)=x$.
A function $G$ on $M\times M$ 
{\em decreases geodesically}, if 
for each $x\in M$ the function $G(x,\cdot)$ decreases geodesically about 
$x$.
\end{deff*}

As concrete examples, we consider the heat kernel on
flat tori of the form $M=\RR^2/\Lambda$, where $\Lambda$ is 
a planar lattice of rank two. 
It was shown in~\cite[final remarks]{ACMM} that geodesic monotonicity 
holds on the tori $a\SS^1\times b\SS^1$, which correspond to 
the rectangular lattices spanned by 
$\{(2\pi a,0), (0,2\pi b)\}$.
We will show that geodesic monotonicity also holds on the
torus defined by a regular triangular lattice
--- a remarkable exceptional case,
which is neither rotationally symmetric nor a product space. 
It turns out that except for  isometries and dilations,
there are no other flat tori with geodesic monotonicity:

\begin{theorem} 
 \label{thm-flat} Let $\Lambda\subset\RR^2$ be a lattice of rank
two, and let $K_t(x,y)$ be the heat kernel
on the flat torus $\RR^2/\Lambda$.

\begin{enumerate}
\item If $\Lambda$ is a regular triangular lattice, 
then $K_t$ is strictly geodesically decreasing
for every $t>0$.

\item Conversely, if $K_t(x,y)$ is geodesically
decreasing for some sequence of times $t_k\to\infty$, then
$\Lambda$ is either a rectangular lattice, or a regular triangular lattice.
\end{enumerate}
\end{theorem}

Continuing in that direction, we will see below that the heat kernel on
any flat Klein bottle cannot be geodesically
decreasing for large values of~$t$. These examples suggest 
that geodesic monotonicity of the heat
kernel constrains the topology of the manifold, as well as the metric.
The remainder of the paper is dedicated to identifying some of the constraints.

From now on, we restrict the discussion
to compact, connected Riemannian manifolds $(M,g)$ whose metric is 
smooth (enough). We prove that for a general metric on any given
manifold, geodesic monotonicity of the heat 
kernel cannot be expected 
at large times,
not even in a small neighborhood of a typical point $x$ (cf. the 
proof of  Theorem \ref{thm-simple}). 

We next introduce a basic technique 
that will be further developed below, and 
prove a simple result.
The heat kernel on a compact Riemannian
manifold $(M,g)$ can be decomposed using 
the natural notion of {\em harmonics}, namely, 
the eigenfunctions of the Laplace-Beltrami 
operator~$-\Delta_g$. Its spectrum consists of a 
sequence of nonnegative eigenvalues
$$0=\lambda_0<\lambda_1\le \lambda_2\le \ldots
$$
repeated according to multiplicity.
We take the eigenfunctions to be an orthonormal sequence $(\phi_j)_{j\ge 0}$ of real-valued functions in $L^2(M,d\Vol_g)$. Since the Laplace-Beltrami operator is the infinitesimal generator of the heat semigroup, the heat kernel has the
spectral expansion 
\begin{equation}
\label{eq:spectral}
K_t(x,y)=
\sum_{j=0}^\infty e^{-\lambda_j t} \phi_j(x)\phi_j(y)\ >\ 0\,.
\end{equation}
As a consequence of 
general pointwise bounds satisfied by the eigenfunctions, 
combined with Weyl's asymptotic law for the eigenvalues,
the series converges absolutely for all $x,y\in M$ and all $t>0$ (see Section~\ref{sec-aux} for more details).

\begin{theorem}
\label{thm-simple}
Let $(M,g)$ be a compact connected Riemannian manifold such that the first eigenvalue of $-\Delta_g$ is simple. Then for all sufficiently
large $t$, the heat kernel fails to be
geodesically decreasing.
\end{theorem}

\begin{proof} 
Let $\phi_1$ be a normalized eigenfunction for $\lambda_1$. 
Since $\phi_1$ is continuous, non-constant, and integrates
to zero over $M$,
its range is an interval that contains $0$
in its interior.  Choose $x,y\in M$ with
$0<\phi_1(x)<\phi_1(y)$.
Then
$$
\phi_1(x)\bigl(\phi_1(y)-\phi_1(x)\bigr)\ >\ 0\,.
$$
Using that the eigenfunction $\phi_0$ associated with
$\lambda_0=0$ is constant, and that
$\lambda_2>\lambda_1$, we obtain from
the spectral expansion that
$$
e^{\lambda_1 t}
\left(K_t(x,y)-\phi_0^2\right) \ \to\ \phi_1(x)\phi_1(y)
$$
as $t\to\infty$, and correspondingly for $K_t(x,x)$. 
For $t$ sufficiently large, it follows that
$$
K_t(x,y)-K_t(x,x)
\ \ge \ \frac12 e^{-\lambda_1 t}\phi_1(x)\bigl(\phi_1(y)-\phi_1(x)\bigr)\ >\ 0\,.
$$
On a minimal geodesic $\gamma$ with $\gamma(0)=x$ and $\gamma(1)=y$, this means that
$K_t(x,\gamma(1))>K_t(x,\gamma(0))$
which shows that $K_t(x,\gamma(s))$ is not monotone decreasing in $s$.
\end{proof}

The hypothesis of Theorem~\ref{thm-simple} is
not empty by a well-known result due to
Uhlenbeck~\cite[Theorem 8]{U} on the generic simplicity of
eigenvalues. The result says that on any compact manifold $M$ of 
dimension $n>1$, the metrics $g$ for which all eigenvalues
of $-\Delta_g$ are simple form a residual subset of 
$C^r$ for every $r\ge n+4$.

Beyond Theorem~\ref{thm-simple},
we find that monotonicity of the heat 
kernel  implies conditions on the eigenspace
associated with the principal eigenvalue,
depending on its multiplicity. 
Indeed, if the first eigenvalue has  multiplicity two 
we have the following rigidity result.

\begin{theorem} \label{thm-product}
Let $(M,g)$ a compact connected 
Riemannian manifold such that the principal eigenvalue of $-\Delta_g$ has 
multiplicity two. If the heat kernel $K_t$
is geodesically decreasing for some sequence of
times $t_k\to\infty$,
then, up to isometry,  $M$ is a product
space $N\times \lambda_1^{-1/2}\SS^1$, 
where $N$ is a connected totally geodesic submanifold of codimension one in $M$.
\end{theorem}

In particular, among two-dimensional 
manifolds where $\lambda_1$ is a double 
eigenvalue,
only  flat tori of the form $a\SS^1\times b\SS^1$ have
geodesically decreasing heat kernels.

The proof of the theorem implies that the nodal sets 
of a principal eigenfunction on $M$ has exactly two connected 
components, both isometric to $N$. By separation
of variables, the heat kernel on $M$ is 
the product of the heat kernels on
$N$ and $\lambda^{-1/2}\SS^1$, and
the spectrum of $-\Delta_g$ is the sum of 
the spectra for the factors.
Moreover the heat kernel on $(N,g)$ 
is geodesically decreasing, and the principal eigenvalue of $-\Delta_g$ on $N$ exceeds
$\lambda_1$. 

Let us mention that manifolds satisfying the conclusion of the theorem are quite rare. We refer 
to \cite{murphyetal} for historical  details and a proof that 
generic manifolds do not have nontrivial totally geodesic submanifolds.

The starting point for the proof of
Theorem~\ref{thm-product} is again the spectral expansion. If $\lambda_1$ has multiplicity
$m$, then
\begin{equation}
\label{eq:limit}
e^{\lambda_1 t}\bigl(K_t(x,y)-\phi_0^2\bigr)\ \to \ 
\sum_{j=1}^m \phi_j(x)\phi_j(y)=:P_{\lambda_1}(x,y)
\end{equation}
uniformly for $x,y\in M$ (see Section~\ref{sec-aux}). The sum on the right hand side is the integral 
kernel of the spectral projection associated with the principal eigenvalue.
If $K_t$ is geodesically decreasing
for some sequence of times $t=t_k\to \infty$, 
then the spectral projection is geodesically 
decreasing as well. We will show in Lemma~\ref{lem-principal-m} that
this monotonicity implies
\begin{equation}
\label{eq:constant}
\phi_1(x)^2+\cdots+\phi_m(x)^2=\frac{m}{\Vol(M)}
\,. \end{equation}
When $m=2$, the joint level sets of the eigenfunctions $\phi_1,\phi_2$ form a foliation of $M$ that we use to construct the claimed isometry
(see equation~\ref{eq:level-theta}).

In known cases where monotonicity holds, all non-trivial eigenvalues have large multiplicity. For instance, on the 
standard sphere the $\ell$-th eigenspace consist of the spherical harmonics of order~$\ell$, whose dimension grows polynomially with $\ell$.
On the real projective space $\RR\PP^n$ the $\ell$-th eigenspace consists of spherical harmonics of order $2\ell$. On tori $a\SS^1\times b\SS^1$, the multiplicity of all nontrivial eigenvalues is even.
This might be connected to the following necessary condition.

\begin{theorem} 
\label{thm-constant}
Let $(M,g)$ be a compact connected Riemannian manifold whose heat kernel is geodesically decreasing for some non-repeating positive sequence $(t_k)_{k\ge 1}$
of times that
does not converge to zero.
Then equation~\eqref{eq:constant} holds whenever $\phi_1,\ldots,\phi_m$ is an orthonormal basis for the eigenspace of some positive eigenvalue
of $-\Delta_g$.
\end{theorem}

The conclusion of~Theorem~\ref{thm-constant} implies
that {\em all} positive eigenvalues of $-\Delta_g$ 
have multiplicity greater than one, as follows.
If $\lambda$ is a simple eigenvalue, then the corresponding 
eigenfunction is constant by equation \eqref{eq:constant}.
Since the constants form the 
eigenspace for $\lambda_0=0$,  the only simple eigenvalue 
of $-\Delta_g$ is the trivial one. 
Thus
Theorem~\ref{thm-constant} implies Theorem~\ref{thm-simple}.

However, under the hypothesis of Theorem~\ref{thm-constant}, if $\lambda$ is a double eigenvalue
we do not recover the conclusion of Theorem~\ref{thm-product} that $M$ is {\em isometric} to a product. Instead, we only obtain a diffeomorphism between $M$ and a mapping torus, see Proposition~\ref{prop-double}. 

The paper is organized as follows. Sections~\ref{sec-honeycomb} and~\ref{sec-flat} are devoted to flat surfaces and the proof of Theorem~\ref{thm-flat}.
In Section~\ref{sec-aux} we derive necessary conditions for geodesic monotonicity  and prove Theorem~\ref{thm-constant}. In Section~\ref{sec-double}, we discuss double eigenvalues in detail, and prove Theorem~\ref{thm-product}.
%%%

%%%%%%%%%%%%%%%%%%%%%%%%%%%%%%%%%%%%%%%%%%%%%%%%%%%%%%%%%%%%%%%%%%%%%%
\section{\bf The new example}
\label{sec-honeycomb}
%%%%%%%%%%%%%%%%%%%%%%%%%%%%%%%%%%%%%%%%%%%%%%%%%%%%%%%%%%%%%%%%%%%%%%

In this section, we establish monotonicity of the heat kernel on the flat torus $\RR^2/\Lambda$,
where $\Lambda$ is the regular triangular lattice generated by
$e_1=(1,0)$ and  $\zeta=(-\frac12,\frac{\sqrt{3}}{2})$, see Figure~\ref{fig-honeycomb}.
The Voronoi cell of the origin
$$
C:=\Bigl\{x\in\RR^2\ \Big\vert\  |x|= \min_{\ell\in L} |x-\ell|\Bigr\}
$$
is a fundamental region for the translations in $\Lambda$.
Every line segment in $C$ emanating from the origin
corresponds to a minimal geodesic in the torus; the boundary of $C$ corresponds to the cut locus of the origin.

By symmetry, $C$ is a regular hexagon, formed by the
intersection of the strips $\{|x\cdot e_1|\le \frac12\}$,
$\{|x\cdot \zeta|\le \frac12\}$, and  $\{|x\cdot \bar \zeta|\le \frac12\}$
where $\bar\zeta=(-\frac12,-\frac12\sqrt{3})$.
In reference to the tiling of the plane by translates of $C$, we call $\Lambda$ the {\em honeycomb lattice}
and $\RR^2/\Lambda$ the {\em honeycomb torus}.
The honeycomb torus can also be obtained by identifying 
opposite sides of $C$. 

Here is the main result of this section.
\begin{prop}\label{prop-honeycomb}
The heat kernel on the honeycomb torus is geodesically decreasing
for every $t>0$.
\end{prop}

We are grateful to an anonymous referee for 
pointing out a related result of Baernstein that affirms that 
for each \mbox{$t>0$} the heat kernel
$K_t(x,y)$ attains its global minimum at  
any pair of points of maximal 
distance~\cite[Theorem 1]{B-1997}. Although our result seems stronger, Baernstein's proof already contains as a key step the reduction we do
 in the proof of Proposition~\ref{prop-honeycomb} below. 
 Starting from an explicit representation for the heat kernel 
 on the torus, Baernstein 
 combines a convexity argument with the maximum principle to show 
 that certain directional derivatives are positive. This is in fact what we achieve from geometric considerations for the heat
 equation alone, without referring to special identities for the solution.

We follow Alonso-Or\'an, Chamizo, Mas, and 
Mart\'inez~\cite{ACMM} and directly apply the
parabolic maximum principle to the heat kernel. 
We will need the strong form
(cf. \cite[Theorem 4 (ii) on p.~54]{E}), paraphrased
here for the convenience of the reader: 

\noindent{\bf Strong maximum principle.}\ 
{\em Let $v$ be a classical solution of the heat
equation $\partial_tv=\Delta v$ on a cylinder $U\times (0,T)$,
where $U$ is a connected domain and $T>0$.
Assume that $v$ is continuous on the closed cylinder
$\overline{U}\times[0,T]$.  If $v$ attains its maximum at a point
$(x^*,t^*)$ with $x^*\in U$ and $t^*>0$, then 
$$ v(x,t)= v(x^*,t^*)$$
for all $x\in \overline U$, $0\le t\le t^*$.}

\begin{proof} [Proof of Proposition~\ref{prop-honeycomb}]
 Let $\Lambda$ be the 
lattice generated by the basis 
$\{e_1,\zeta\}$.
We need to show that $K_t(\gamma(0),\gamma(s))$ 
is nondecreasing along any minimal geodesic $\gamma$.
By translation, we may assume that $\gamma(0)=(0,0)$.

We consider $K_t((0,0),x)$ as a function
on $\RR^2$ that is $\Lambda$-periodic,
$$
K_t((0,0),x+\ell)=K_t((0,0),x)\,,\qquad (x\in\RR^2,\ell\in \Lambda)\,.
$$
We approximate it by a smooth solution
of the heat equation $\partial_t u=\Delta u$,
with initial values $u(x,0)=u_0(x)$
that are symmetric under all isometries of 
$\RR^2$ mapping the honeycomb lattice $\Lambda$ to itself. These symmetries 
are generated by the reflection symmetries of $C$, combined with the 
lattice translations. By uniqueness of the solution of this initial-value problem, $u(\cdot, t)$ shares these symmetries.
We will prove that
\begin{equation}
x\cdot\nabla_x K_t((0,0),x)< 0
\label{eq:strict}
\end{equation}
for all $x\ne 0$ in the interior of $C$ and all $t>0$. In the proof, we first establish monotonicity,
and then revisit the argument to obtain strict monotonicity.

\begin{figure}
 \centering\includegraphics{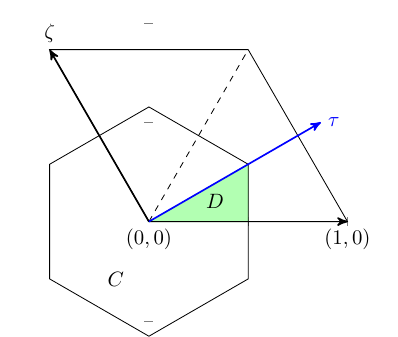}
 
\caption{\small Two fundamental domains for the standard honeycomb lattice. The parallelogram is spanned by the basis vectors $e_1=(1,0)$ and $\zeta=(-\frac12,\frac12\sqrt{3})$.
The boundary of the
regular hexagon $C$
corresponds to the cut locus of the origin in the honeycomb torus
$\RR^2/\Lambda$. The lattice is symmetric under reflection across the three lines that form the sides of $D$.}
    \label{fig-honeycomb}
\end{figure}

Let $D=\{x\in C: 0\le \text{arg}\,x\le \frac{\pi}{6}\}$ 
be the small right-angled triangle
shaded green in Figure~\ref{fig-honeycomb}.
This is a fundamental region for
the action of 
the
symmetries of the lattice
on the plane.
Taking advantage of these symmetries, it 
suffices to consider $x\in D$.

As in ~\cite[Section 2]{ACMM}, we apply the maximum principle
to certain directional derivatives of $u$. The partial derivatives
$v_1:=\partial_{x_1}u$ and $v_2:=\partial_{x_2}u$ satisfy
the heat equation on $\RR^2$. 
By symmetry of $u$, the function $v_1$ vanishes on the vertical 
edge of $D$ and is reflection-symmetric
across the horizontal edge. Similarly, $v_2$ 
vanishes on the horizontal edge and
is reflection-symmetric across the vertical edge. On the diagonal edge that forms the hypotenuse of $D$,
the spatial gradient
$\nabla u=(v_1,v_2)^t$ is tangent to $D$ due to the
reflection symmetry of~$u$. 
Therefore, on this edge,
the partials are in the ratio $v_1 : v_2=\sqrt{3}:1$.
Clearly, $v_1$ and $v_2$ vanish at the vertices of $C$ and at the midpoints of edges, as well as the origin, 
all of which are critical points for~$u$.

Our first claim is that if 
$\partial_{x_1}u_0\le 0$ and $\partial_{x_2}u_0 \le 0$
then $v_1\le 0$ and $v_2\le 0$ on $D\times[0,\infty)$. 

Suppose that
$v_1$ attains its maximum over the compact set $D\times [0,T]$ at a certain point $(x^*,t^*)$. We will argue by contradiction that $t^*=0$. If instead $t^*>0$,
then, by the strong maximum principle, $x^*$ lies on the boundary of~$D$.  
We know that it lies in the interior of
some edge, because $u$ has critical 
points at the origin and at the vertices of $D$.
Let $D'$ be the union of $D$ and its reflection
across the horizontal edge. 
Since $v_1$ is symmetric under this reflection and vanishes on the vertical edge, it attains its maximum
on both diagonal edges of $D'$.
Thus $x^*$ lies on the
diagonal edge of $D$ that it shares with $D'$.   

Conversely,                              
         if $v_2$  attains its maximum at 
a point $(x^{**},t^{**})$ with $t^{**}>0$,
then
by the same reasoning as for $v_1$ (but reflecting
across the vertical edge of $D$) we see that
$x^{**}$ lies on the diagonal edge of $D$ as well.

By maximality of $v_1(x^*,t^*)$, 
$$
v_1(x^*,t^*)\ \ge \ v_1(x^{**},t^{**}) \ =\  \sqrt{3}v_2(x^{**}, t^{**})\,,
$$
where we have used the fixed ratio between $v_1$ and $v_2$ 
on the diagonal edge.  On the other hand, 
by maximality of $v_2(x^{**},t^{**})$,
$$
v_1(x^*,t^*)\ =\ \sqrt{3}v_2(x^*,t^*)\ \le \sqrt{3}
v_2(x^{**}, t^{**})\,.
$$
We conclude that both inequalities hold with
equality, and $v_1$ and $v_2$ simultaneously
attain their maxima over $D\times[0,T]$ at $(x^*, t^*)$ (and
also at $(x^{**}, t^{**})$).

Let $\tau=(\frac12\sqrt{3},\frac12)$ be the unit
vector along the diagonal edge. 
Since both components of $\tau$ are positive,
the directional derivative
$$v:=\tau\cdot \nabla u = \frac12\sqrt{3} v_1+
\frac12 v_2$$
also attains  its maximum over $D\times [0,T]$
at $(x^*,t^*)$. Consider the quadrilateral
$D''$ formed by the union of $D$ and its reflection 
across the diagonal edge.  Since $v$ is symmetric
under reflection at the diagonal edge of $D$
this means that $v(x^*,t^*)$ is the maximum for
$v$ on $D''\times[0,T]$.
But $x^*$ is an interior point for $D''$,
contradicting the strong maximum principle.

This contradiction shows that necessarily $t^*=0$.
 
As a consequence 
$v_1,v_2$ are nonnegative  on~$D$ for all $t\ge 0$.
Since the coordinates of any point $x
\in D$ are nonnegative, this implies that $x\cdot\nabla u(x,t)\le 0$ on $D$, 
proving the first claim.

\smallskip
To conclude the proof of monotonicity we approximate $K_t((0,0),x)$ 
by the solution of the heat equation with
initial data $$u(x,0)= \sum_{\ell\in \Lambda} u_0(x+\ell),$$
where $u_0$ is a radially decreasing function
of unit total mass that is supported on a small disk about the origin
in the interior of $C$.
Then $\partial_{x_1}u_0\le 0$ and $\partial_{x_2}u_0\le 0$
on $D$, and neither vanishes identically.
Therefore our previous arguments apply, and the solution satisfies
$x\cdot \nabla u(x,t)<0$ for $x\ne 0$ in $D$ and
all $t>0$.  We let the initial condition $u_0$ converge
to a point measure to obtain the monotonicity
property for the heat kernel. 

\smallskip
For strict monotonicity, assume additionally that
$\partial_{x_1}u_0$ and $\partial_{x_1}u_0$ do not vanish identically on $D$,
i.e., $v_1$ and $v_2$ have  non-trivial initial values. We already know that $v_1, v_2\le 0$ on $D$ for all $t\ge 0$.
We claim that for $t>0$, $v_1$ vanishes on $D$ only on the vertical edge and at $x=0$, 
and $v_2$ vanishes only on the horizontal edge and at the upper vertex.

Suppose that $v_1(x^*,t^*)=0$ for some $t^*>0$.
Since $v_1(\cdot,0)$ is not identically zero,
by the strong maximum principle $x^*$ cannot lie
in the interior of $D$. As in the proof of the first claim, by considering reflections across the horizontal and diagonal edges, we see that it cannot lie on the interior of those edges, either.

Therefore $v_1$ vanishes only on
the vertical edge and at $0$.
In  the same way
we find that $v_2$ vanishes
only on the horizontal edge and at the upper vertex.
It follows that $x\cdot \nabla u(x,t)<0$
for all $t>0$ and all $x\in D$ except
at its vertices.

\smallskip
To prove strict monotonicity for $K_t((0,0),x)$, we recall that
$$
\partial_{x_1}K_t((0,0),x)\le 0 \quad\text{and}\quad 
\partial_{x_2}K_t((0,0),x)\le 0 
$$
on $D$, for all $t>0$. Moreover, the partial derivatives
$\partial_{x_1}K_t((0,0),x)$ and $\partial_{x_2}K_t((0,0),x)$
cannot vanish identically on $D$, since for each $t>0$ the heat kernel
is a non-constant smooth function of $x$
that is symmetric under reflection across the sides of $D$.  Fix $\eps>0$.
By the semigroup property,
$u(x,t):= K_{t+\eps}((0,0),x)$ solves the heat equation. Since 
its initial values $K_\eps((0,0),x)$ are non-increasing and non-constant in both variables on $D$, it 
follows that $x\cdot\nabla u(x,t)<0$ for $t>\eps$, for 
all $x\in D$ except the vertices. We finally take $\eps\to 0$ to obtain equation~\eqref{eq:strict} on $D$, and hence $C$, for all $t>0$.
\end{proof}

It is clear from the proof of Proposition~\ref{prop-honeycomb} 
that for each $t>0$,
the heat kernel $K_t((0,0),x)$ has a 
maximum at $x=0$, minima at the vertices of $C$, saddle points
at the midpoints of the edges, and no other critical points on~$C$.
This corresponds to one local maximum, two local minima,
and three saddle points on the honeycomb torus $\RR^2/\Lambda$,
We confirm that their indices add up
to $1-3+2=0$, the Euler characteristic of the torus.

%%%%%%%%%%%%%%%%%%%%%%%%%%%%%%%%%%%%%%%%%%%%%%%%%%%%%%%%%%%%%%%%%%%%%%
\section{\bf Non-monotonicity on flat surfaces}\label{sec-flat}
%%%%%%%%%%%%%%%%%%%%%%%%%%%%%%%%%%%%%%%%%%%%%%%%%%%%%%%%%%%%%%%%%%%%%%

\subsection{Tori}

In this section, we investigate geodesic monotonicity
of the heat kernel on  flat tori of the
form $\RR^2/\Lambda$, where $\Lambda$ is a lattice of 
rank two, and prove Theorem~\ref{thm-flat}.
The heat kernels on such tori can be expressed
in terms of the heat kernel on $\SS^1$, though their monotonicity properties are not obvious from the formulas (see, e.g.~\cite[p. 231]{B-1997}).
We need some explicit information about the spectrum of
the Laplacian on such surfaces,
summarized below.

Let us recall that a planar lattice $\Lambda$ is a discrete subgroup
of $\RR^2$, which acts on itself by translation. The lattice has rank two, if it
consists of all integer
linear combinations of a pair of linearly independent
vectors. Such a pair is called
a {\em basis} for $\Lambda$. Since geodesic monotonicity is invariant under scaling, we assume  without loss of generality that the
length of the smallest element in the basis is normalized to one. Taking advantage of rigid rotations
and a change of basis, we may also assume
that the basis vectors are
$(1,0)$ and $(-a,b)$, where
$0\le a\le \frac12$, $b> 0$, and $a^2+b^2\ge 1$.
With this choice of parameters,
$ (1,0)$ minimizes the modulus among non-zero elements of $\Lambda$, and $(-a,b)$ minimizers the modulus among elements linearly 
independent of $(1,0)$. The rectangular lattices correspond to 
the parameter values $a=0$, $b\ge 1$, 
and the standard honeycomb lattice corresponds to  $a=\frac12$, $b=\frac12\sqrt{3}$. 

The spectrum of the Laplacian on $\RR^2/\Lambda$ is conveniently
expressed in terms of the {\em dual lattice} $\Lambda'$, 
consisting of all points $\ell\in\RR^2$
such that $\ell\cdot x$ is an integer for every $x\in \Lambda$.
An orthogonal set of eigenfunctions
is given by
$$
\phi_\ell(x)= e^{2\pi i \ell\cdot x}\,, \qquad (\ell\in \Lambda')\,.
$$
Notice that these are orthogonal in $L^2$ but should be 
normalized by a factor $b^{-1/2}$, to compensate for the area spanned by the basis vectors. 
The corresponding eigenvalues
are $\lambda_\ell=(2\pi|\ell|)^2$.
Since $\pm\ell\neq 0$ give rise to the
same eigenvalue, the multiplicity of
every non-trivial eigenvalue is even,
and on the corresponding eigenspace, any orthonormal basis satisfies equation~\eqref{eq:constant}. 

The basis of $\Lambda'$ dual to 
$\{(1,0), (-a,b)\}$
 consists of the vectors $v_1~=~(1,b^{-1}a)$
and $v_2=(0, b^{-1})$. Here, $v_2$ minimizes the
modulus among non-zero elements of $\Lambda'$, and $v_1$ minimizes 
the modulus among elements linearly independent of $v_2$.
Incidentally we have proved:

\begin{lemma}\label{lemaprincipal}
Let $\Lambda$ be the planar lattice generated by the
basis vectors $(1,0)$ and $(-a,b)$. If $0\le a\le 1$ and $a^2+b^2>1$, then the principal eigenvalue of the Laplacian on $\RR^2/\Lambda$ is $\lambda_1=(2\pi/b)^2$.
\end{lemma}

\noindent We group the flat tori according to their lattice symmetries as follows.

{\em Lattices without special symmetries.}
If $0< a\le \frac12$ and $a^2+b^2>1$, then
the principal eigenvalue has multiplicity two. The corresponding eigenspace is spanned by
$\cos(2\pi x_2/b)$ and $\sin (2\pi x_2/b)$. 

{\em Isosceles lattices.}
If $0< a<\frac12$ and $a^2+b^2=1$, then
both basis vectors have unit length. 
This means that $\Lambda$ has additional symmetries, 
associated with the rhombus spanned by the basis vectors; these
symmetries manifest themselves as reflections on the 
torus $\RR^2/\Lambda$. The multiplicity of the principal eigenvalue is four. The eigenspace is spanned by 
$\cos(2\pi x_2/b)$, $\sin(2\pi x_2/b)$, 
$\cos(2\pi(bx_1+ax_2)/b)$, and $\sin(2\pi(bx_1+ax_2)/b)$.

{\em Rectangular and square lattices.} For $a=0$ and $b>1$, the principal eigenvalue has multiplicity two, and the eigenspace is spanned by $\cos(2\pi x_2/b)$ and $\sin (2\pi x_2/b)$. 
For the square lattice, where $a=0$ and $b=1$, the multiplicity of the principal eigenvalue is four. The eigenspace is spanned by $\cos(2\pi x_1)$, $\sin(2\pi x_1)$, $\cos(2\pi x_2)$, and $\sin (2\pi x_2)$.

{\em The honeycomb lattice.}
If $a=\frac12$ and $b=\frac12\sqrt{3}$, then
the points closest to zero in $\Lambda$ form 
a regular hexagon, and similarly for $\Lambda'$.
The multiplicity of the principal eigenvalue is six, and the eigenspace is
spanned by 
$\cos(2\pi x_2/b)$, $\sin(2\pi x_2/b)$, 
$\cos(2\pi(bx_1\pm ax_2)/b)$, and $\sin(2\pi(bx_1\pm ax_2)/b)$.

\noindent{\bf Remark.} On the honeycomb torus,
all  non-trivial eigenvalues of $-\Delta$
have multiplicity divisible by six.
The reason is that the dual lattice is 
symmetric under rotation by $\pi/3$, and 
the orbit of every non-zero point in the
plane under this rotation has length six.

\smallskip
With this classification it is clear that 
for the proof of Part (2) of Theorem~\ref{thm-flat}
we can focus on the first two cases. Although similar arguments apply, they need to be treated separately, as follows.

\begin{prop}[Lattices without special symmetries]
\label{prop-generic}
Let $\Lambda$ be the planar lattice generated by the basis vectors
$(1,0)$ and $(-a,b)$. If
$0<a\le\frac12$ and $a^2+b^2>1$ then for large $t$
the heat kernel on $\RR^2/\Lambda$ is not geodesically decreasing.
\end{prop}

\begin{figure}
 \centering
 \includegraphics{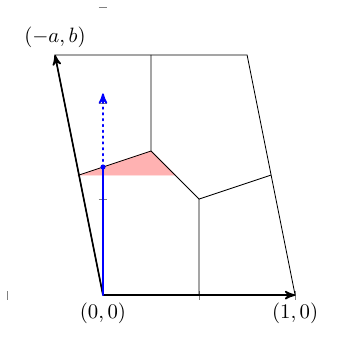}
\caption{ 
   \small
   Fundamental parallelogram of a lattice $\Lambda$ without special symmetries. The lattice basis  is $\{(1,0),(-a,b)\}$, where $0<a\le\frac12$, $b>0$, and
$a^2+b^2>1$. The thin lines correspond to the cut locus of the origin in the torus $\RR^2/\Lambda$. 
Geodesic monotonicity fails on the segment of the vertical geodesic emanating
from $(0,0)$ that lies inside the shaded triangle.}
    \label{fig-generic}
\end{figure}

\begin{proof} We will show that monotonicity fails
along a certain minimal geodesic emanating from the origin.
Let $\gamma(s)=(0, sb)$
for $0\le s\le 1$, as in the Figure~\ref{fig-generic}.
Since the distance of $(0,\frac{b}{2})$ 
from any non-zero lattice point strictly exceeds $\frac{b}{2}$,
the geodesic is minimal up to 
some $s^*>\frac12$. Explicitly,  the geodesic meets the cut locus at
$(0,\frac12(a^2+b^2))$, 
which is equidistant to the origin and to $(-a,b)$.

The spectral projection onto the eigenspace for $\lambda_1$ has integral kernel
$$
P_{\lambda_1}(x,y)
=\frac{2}{b} \,\cos(2\pi (x_2-y_2)/b)\,.
$$
Since 
$P_{\lambda_1}((0,0),\gamma(s)) = 2b^{-1} \cos(2\pi s)$
increases for $\frac12\le s\le s^*$, geodesic monotonicity fails.
\end{proof}

\begin{prop}[Isosceles lattices]
\label{prop-isosceles}
Let $\Lambda$ be the planar lattice 
generated by the basis vectors
$(1,0)$ and $(-a,b)$.
If $0<a<\frac12$ and $a^2+b^2=1$, then for large $t$
the heat kernel on $\RR^2/\Lambda$ is not geodesically decreasing.
\end{prop}

\begin{figure}[t]
 \centering
 \includegraphics{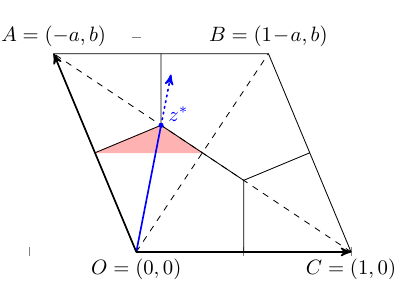}
\caption{
\small Fundamental parallelogram 
of an isosceles lattice $\Lambda$ with basis $\{(1,0),(-a,b)\}$, where $0<a<\frac{1}{2}$, $b>0$, and 
$a^2+b^2=1$. In addition to translation along the basis vectors, the lattice is symmetric under reflection at 
the diagonals (dashed), both of which are longer than the sides.
The thin lines inside the parallelogram correspond to the cut locus of the origin in $\RR^2/\Lambda$. The point $z^*$ is equidistant
to $O$, $A$, and $B$.
Geodesic monotonicity fails on the segment
of the geodesic connecting $O$ to $z^*$ that lies inside the 
small shaded triangle.}
    \label{fig-isosceles}
\end{figure}

\begin{proof} As in the proof of Proposition~\ref{prop-generic},
we show that the projection $P_{\lambda_1}(x,y)$ violates
monotonicity along some geodesic emanating from the origin, see Figure~\ref{fig-isosceles}. By Lemma~\ref{lemaprincipal} the principal eigenvalue 
is $\lambda_1=(2\pi/b)^2$, with multiplicity four as we are dealing with isosceles lattices according to our previous classification. The kernel of the projection onto the eigenspace of $\lambda_1$ simplifies to
$P_{\lambda_1}(x,y)=P(x-y)$, where
$$
P(x) = \frac{2}{b}\Bigl(\cos(2\pi x_2/b) +
\cos(2\pi (bx_1+ax_2)/b) \Bigr)\,.
$$
We will show that $P$ increases on the final segment of a certain minimal geodesic emanating from the origin.

Denote the vertices of the fundamental parallelogram by 
$O=(0,0)$, $A=(-a,b)$, 
$B=(1-a,b)$, and  $C=(1,0)$, and let
$z^*$ be the center of the circle through $O$, $A$, and $B$.
The three points form an isosceles
triangle with an acute angle at 
the apex $A$. Hence $z^*$ lies in the interior of the triangle, on the 
long diagonal that joints $C$ to $A$. 
By construction, $z^*$ lies on the cut locus of the origin,
and the geodesic parametrized by $\gamma(s)=sz^*$ 
is minimal for $0\le s\le 1$. We claim that
\begin{equation}
\label{eq:claim-isosceles}
\frac{d}{ds}P(\gamma(s))\Big\vert_{s=1}
= z^*\cdot \nabla P(z^*)>0\,.
\end{equation}
It follows that $P$ increases on a small segment of the
geodesic shortly before it reaches~$z^*$, 
which contradicts monotonicity.

To see the claim, parametrize the long diagonal 
as $\beta(s)=(1,0)+ s\xi$, where 
$\xi=(-1\!-\!a,b)$ is the direction vector from $C$ to $A$. Along the diagonal, $$P(\beta(s))=\frac{4}{b}\cos(2\pi s )\,.$$
Since $P$ is symmetric under reflection at the
diagonal, $\nabla P(\beta(s))$ is tangent
to $\beta$.  Since $P(\beta(s))$ increases for $s\in[\frac12,1]$, the gradient is
a positive multiple of $\xi$ on that interval.

Let $\eta=(1-a,b)$ be the direction vector
for the short diagonal that joins $O$ to
$B$. Evidently, 
$A=\frac{1}{2}\xi+\frac{1}{2}\eta$,
and $B=\eta$. Since $z^*$ is in the convex hull of $O$, $A$ and $B$ we can write $z^*$ in the  basis $(\xi,\eta)$ as $z^*=r\xi+t\eta$ for some $r,t\in(0,1)$. Recalling that $z^*$ lies on the segment of the long diagonal
where $\nabla P$ is a positive multiple of $\xi$, we see that
$$z^*\cdot\nabla P(z^*)=(r\xi+t\eta)\cdot \nabla P(z^*)
=r\,|\xi|\,|\nabla P(z^*)| >0\,.
$$
In the last step we have
used that $\xi\cdot\eta=0$ because the fundamental parallelogram forms a rhombus, whose diagonals are orthogonal.
\end{proof}

\subsection{Klein bottles}

Let us consider the one-parameter family of flat Klein bottles given by
\begin{equation}
\label{eq:def-K}
\mathcal{K}= \RR^2\Big\slash \Bigl\{(x_1,x_2)\sim(1+x_1,x_2)\sim(1-x_1,b+x_2)\Bigr\}\,.
\end{equation}
A crucial difference between $\mathcal{K}$ and a flat torus is that
$\mathcal{K}$ is symmetric under translation along the second coordinate 
axis, but not  the first. Consider the orientable double 
cover of $\mathcal{K}$, defined by $\mathcal{T}=\RR^2/\Lambda$, where
$\Lambda$ is the rectangular lattice with basis $\{(1,0),(0, 2b)\}$. 
We view $\mathcal{K}$ as the quotient of $\mathcal{T}$ under the isometry 
$(x_1,x_2)\mapsto (1-x_1,b+x_2)$.

A complete set of orthogonal eigenfunctions for $\mathcal{K}$ is given by
\begin{equation}
    \label{eq:Klein-eigenfunctions}
\begin{cases}
\cos(2\pi \ell_1 x_1) e^{\pi i\ell_2x_2/b}\,,\qquad
& (\ell_1\ge 0,\ \ell_2\ \text{even})\,,\\
\sin(2\pi \ell_1 x_1) e^{\pi i\ell_2x_2/b}\,,\qquad
& (\ell_1>0,\ \ell_2\ \text{odd})\,.
\end{cases}
\end{equation}
Here, $\ell=(\ell_1,\ell_2)$ is a pair of integers, and as in the case of the tori, 
the eigenfunctions should be normalized by a factor of
$(2/b)^{1/2}$.
The principal eigenvalue is $\lambda_1=
\min\{ (2\pi)^2,(2\pi/b)^2\}$, corresponding 
to $\ell=(1,0)$ or $(0,\pm 2)$.
For $b<1$, the principal eigenvalue is simple,
with eigenfunction $\cos(2\pi x_1)$. For $b>1$ the multiplicity is two, with eigenfunctions $\cos(2\pi x_2/b)$ and $\sin (2 \pi x_2/b)$.
At $b=1$ the eigenvalues cross and the multiplicity is three.

\begin{prop}
\label{prop-K}
Let $\mathcal{K}$ be a flat Klein bottle 
given by equation \eqref{eq:def-K} for some $b>0$. 
Then for sufficiently large $t$, the heat kernel on $\mathcal{K}$ is not geodesically decreasing.
\end{prop}

\begin{figure}[t]
 \centering
 \includegraphics{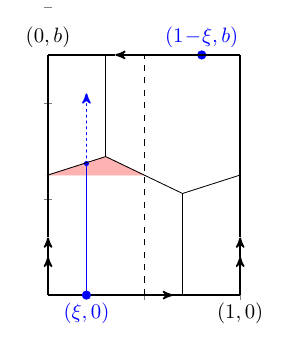}
\caption{
   \small
   Fundamental rectangle for a flat Klein bottle of height $b>1$. The left and right edges are identified to form a cylinder, and the top and bottom circles are glued by the orientation-reversing isometry $\psi(x)=-x$. 
  Given $\xi\in (0,\frac12) $, the thin lines correspond to the cut locus of $(\xi,0)$ in the Klein bottle.
   Geodesic monotonicity fails on the segment of the vertical geodesic emanating from $(\xi,0)$ that
lies inside the shaded triangle.}
    \label{fig-Klein}
\end{figure}

\begin{proof} We distinguish three cases depending on the multiplicity of $\lambda_1$.
If $b<1$, then the principal eigenvalue is simple and monotonicity
fails by Theorem~\ref{thm-simple}.

If $b>1$, the principal eigenvalue has multiplicity two, with normalized eigenfunctions $(2/b)^{1/2}\cos(2\pi x_2/b)$
and $(2/b)^{1/2}\sin(2\pi x_2/b)$. 
The spectral projection is given by
$$
P_{\lambda_1}\bigl(x,y\bigr) 
=\frac{2}{b}\cos(2\pi(x_2-y_2)/b)\,.
$$
As in the proof of Proposition~\ref{prop-generic}, we consider a vertical geodesic
$\gamma(s)=(\xi,sb)$, where $0<\xi<\frac12$, see Figure~\ref{fig-Klein}. Since
the distance from $(\xi,\frac{b}{2})$ 
to $(-\xi,b)$ in $\mathcal{T}$ strictly exceeds $\frac{b}{2}$,
the geodesic is minimal up to 
some $s^*>\frac12$. 
Since $P_{\lambda_1}
\bigl(\gamma(0),\gamma(s)\bigr)
= (2/b)\cos (2\pi s)$ strictly increases for $\frac12\le s\le s^*$,
geodesic monotonicity fails.

If $b=1$, the principal eigenvalue has
multiplicity three, with normalized eigenfunctions
$\phi_1(x)=\sqrt{2} \cos(2\pi x_1)$. 
$\phi_2(x)=\sqrt{2}\cos(2\pi x_2)$, and
$\phi_3(x)=\sqrt{2}\sin(2\pi x_2)$.  The eigenfunctions $\phi_2$ and $\phi_3$ correspond to $(\ell_1, \ell_2)=(0,\pm 2)$, accounting for the extra factor of 2.  
The spectral projection is given by
$$
P_{\lambda_1}(x,y) = 2\cos(2\pi x_1)\cos(2\pi y_1)
+ 2\cos(2\pi (x_2-y_2))\,.
$$
 We consider again the vertical geodesic through $(\xi,0)$, given by
$\gamma(s)=(\xi, s)$.
Since $P_{\lambda_1}\bigl(\gamma(0),\gamma(s)\bigr)
= 2\cos^2 ( 2\pi \xi ) + 2\cos(2\pi s)$ strictly increases
for $\frac12\le s\le s^*$,
geodesic monotonicity fails also in this case.
\end{proof}

\noindent \textbf{Remark.} Proposition~\ref{prop-K} can
be strengthened using
the results of Section~\ref{sec-aux}. Specifically, the conclusion of
Theorem~\ref{thm-constant}
fails for the eigenvalue 
$\lambda=(2\pi)^2$ on $\mathcal{K}$.
Note that this is not the principal eigenvalue unless $b\le 1$.
If $b\ge 1$ is a half-integer,
then $\lambda=(2\pi)^2$ has multiplicity three, 
corresponding to the spectral parameters
$\ell=(1,0)$ and $(0,\pm 2b$).
The corresponding
eigenfunctions are (up to a normalization factor) $\cos(2\pi x_1)$,
$\cos(2\pi x_2)$,
and $\sin(2\pi x_2)$.
Otherwise, $\lambda$ is simple, corresponding
to $\ell=(1, 0)$. 
In any case,
equation~\eqref{eq:constant} does not hold. By Theorem~\ref{thm-constant}, 
the heat kernel on $\mathcal{K}$ is geodesically decreasing at most for a sequence of times 
that is either finite, or converges to zero.

It is an open question whether there exist any time $t>0$ and any $b>0$ 
for which the heat kernel on $\mathcal{K}$ is geodesically decreasing.
The same question can be asked about a torus that is
neither rectangular nor a honeycomb torus.

\begin{proof}[Proof of Theorem~\ref{thm-flat}] By Proposition~\ref{prop-honeycomb}, the heat kernel on the honeycomb torus is geodesically decreasing for all $t>0$, proving the first claim. 

Conversely, suppose the heat kernel is geodesically decreasing
on a flat torus $\RR^2/\Lambda$ for some sequence
of times $t_k\to\infty$. As discussed at the beginning 
of the section, by a suitable dilation, rigid rotation,
and change of basis it suffices to consider the case
where
$\Lambda$ is spanned by $(1,0)$ and $(-a,b)$ for some parameters $a,b$ with
$0\le a\le \frac12$ and $a^2+b^2\ge 1$.
Propositions~\ref{prop-generic}
and~\ref{prop-isosceles} imply that the lattice is either
rectangular (with $a=0$, $b\ge 1$) or the honeycomb lattice
($a=\frac12$, $b=\frac12\sqrt{3}$)
\end{proof}

%%%%%%%%%%%%%%%%%%%%%%%%%%%%%%%%%%%%%%%%%%%%%%%%%%%%%%%%%%%%%%%%%%%%%%
\section{\bf Necessary conditions}
\label{sec-aux}
%%%%%%%%%%%%%%%%%%%%%%%%%%%%%%%%%%%%%%%%%%%%%%%%%%%%%%%%%%%%%%%%%%%%%%

In this section, we study conditions on the
eigenvalues and eigenfunctions 
of the Laplace-Beltrami operator $-\Delta_g$
that follow from monotonicity of the heat kernel. 
To ease notation, we will omit the metric $g$ from the notation
for geometric quantities such as the Riemannian metric,
volume, and gradient, as well a the Laplace-Beltrami operator.
We have already excluded the possibility of a simple principal eigenvalue in the introduction (cf. Theorem \ref{thm-simple}).

The following observation will play a crucial role.

\begin{lemma}\label{lem-const}
Let $(M,g)$ be a connected Riemannian manifold,
and let $G:M\times M\rightarrow \RR$ be a symmetric function of class $C^1$.
If $G$ is geodesically decreasing
then $G(x,x)$
is constant on $M$.
\end{lemma}

\noindent \textbf{Remark.} The differentiability hypothesis 
cannot be improved. Indeed, for $\eps\in (0,1)$ the function $G(x,y)=-|x-y|+\epsilon(x+y)$ 
is geodesically decreasing, but
$G(x,x)=2\eps x$ is not constant.

\begin{proof}
Given $x\in M$, let $B$ be an open ball centered at $x$
whose radius does not exceed the injectivity radius
of $M$ at $x$. Since $x$ can be joined to any point $y\in B$ 
by a unique minimal geodesic, monotonicity implies
that $G(x,x)\ge G(x,y)$. Since $y$ was arbitrary,
$G(x,\cdot)$ attains a local maximum at $y=x$.
By the first-derivative test,
$\nabla_y G(x,y)\big\vert_{y=x}=0$.
Since $G$ is symmetric, the chain rule yields
$$
\nabla_x G(x,x)= 2\nabla_yG(x,y)\Big\vert_{y=x}=0\,,
$$
which, by connectedness,
implies that $G(x,x)$ is constant on $M$.
\end{proof}

We apply this to the spectral projection 
associated with the principal eigenvalue,
as given by equation~\eqref{eq:limit}.

\begin{lemma}
\label{lem-principal-m} 
Let $(M,g)$ be a compact connected Riemannian manifold whose principal eigenvalue $\lambda_1$ has multiplicity $m$, and let $\phi_1,\ldots,\phi_m$ be an orthonormal basis of eigenfunctions for $\lambda_1$.
If the heat kernel on $M$ decreases along geodesics for some sequence of
times $t_k\to\infty$, then $\phi_1^2+\cdots + \phi_m^2$ is constant on $M$.
\end{lemma}

\begin{proof} By equation~\eqref{eq:limit}, the monotonicity of the heat kernel implies
monotonicity of $P_{\lambda_1}$. Therefore, the claim follows from Lemma~\ref{lem-const}. 
\end{proof}

Since $\phi_1,\ldots, \phi_m$ are orthonormal,
the value of the constant in Lemma~\ref{lem-principal-m} is given by
\begin{equation}
\label{eq:def-c}
c^2:= \phi_1^2\ +\ \cdots \ + \ \phi_m^2 =\frac{m}{\Vol(M)}\,.
\end{equation}
We conclude from the lemma that
\begin{equation}
\label{eq:def-Phi}
\Phi (x):=c^{-1}\begin{pmatrix} \phi_1(x),\ldots, \phi_m(x)\end{pmatrix}^t
\end{equation}
defines a map from $M$ to the unit sphere whose components are orthogonal eigenfunctions. At least in two dimensions, these conditions are reminiscent of the properties of $\lambda_1$-extremal metrics~\cite{EI}.

Monotonicity also has consequences for the joint level sets of the principal eigenfunctions.

\begin{lemma} [Convexity]
\label{lem-convex} Under the assumptions of Lemma~\ref{lem-principal-m}, 
the set
$$
M_u:= \bigl\{x\in M : \Phi(x)=u\bigr\}
$$
is geodesically convex for each $u\in \SS^{n-1}$. In particular, $M_u$ 
is connected. Furthermore, if $M_u$ is a submanifold then it is totally geodesic.  
\end{lemma}

\begin{proof}
Fix $u\in \SS^{m-1}$ and $x,y\in M_u$.
Let $\gamma(s)$ be a minimal geodesic
in $M$ starting from $\gamma(0)=x$ and ending at $\gamma(1)=y$. The function
$$
f(s):= u\cdot\Phi(\gamma(s))\,,\qquad 0\le s\le 1
$$
has boundary values $f(0)=f(1)=1$
since $\Phi(x)= \Phi(y)=u$ by assumption. By equation ~\eqref{eq:def-Phi}, we have that $f(s) =  c^{-2}P_{\lambda_1}(x,\gamma(s))$ for some constant $c>0$.
Since this is nondecreasing by the monotonicity
property of $P_{\lambda_1}$, it follows that $f(s)\equiv 1$, and hence
$\Phi(\gamma(s))=u$ for all $s\in [0,1]$.

In the case where $S$ is a submanifold, we need to show that any
geodesic $\gamma(s)$ in $M_u$ (with the induced metric)
is a geodesic in $M$. Let $\rho>0$ be smaller than the injectivity
radius of $M$ and of $M_u$. Consider the case where the diameter of $\gamma$ in $M_u$ (and hence in $M$) is less than $\rho$.
Then the endpoints $x=\gamma(0)$ and $y=\gamma(1)$ 
are connected by a unique minimal geodesic $\eta(s)$ in $M$,
with $\eta(0)=x$, $\eta(1)=y$.
By geodesic convexity, $\eta(s)\in M_u$
for $0\le s\le 1$, and hence $\eta=\gamma$. If $\gamma$ has larger
diameter, we cut it into finitely many segments 
of diameter less than $\rho$, each of which is a minimal geodesic in $M$.
\end{proof}

The remainder of this section concerns higher eigenvalues.
The spectral expansion in 
equation~\eqref{eq:spectral} is
again the main tool. Let us briefly consider the
convergence of the series. The eigenfunctions of
the Laplace-Beltrami operator
satisfy bounds of the form 
$$\|\phi_j\|_{L^{\infty}}\leq C(M,g)\lambda^{\alpha}\,,$$ 
with a dimension-dependent exponent $\alpha$.
For $\alpha=\frac{n}{2}$ these can be obtained
directly from Sobolev's inequalities; the
sharp exponent $\alpha=\frac{n-1}{2}$ is
due to H\"ormander~\cite{Ho} (cf. 
Theorem 84 \cite{Canzani} or \cite{Sogge} for generalizations to $L^p$). On the other hand, Weyl's law implies that
the eigenvalues grow as
$$
\lambda_k \sim c(M,g) k^{\frac{2}{n}}\,.
$$ 

For any $\tau>0$, since the series
$\sum_k k^{(n-1)/n}e^{-tk^{2/n}}$
converges uniformly on $[\tau,\infty)$, the spectral expansion converges absolutely, uniformly
for $x,y\in M$ and $t\ge \tau$. As a consequence, 
the heat kernel is analytic in $t$, and has the same
regularity in $x,y$ as the eigenfunctions. 

In equation~\eqref{eq:spectral}, the eigenvalues are repeated 
according to multiplicity, and different choices can be made 
for the basis of eigenfunctions.  For the purpose of
uniqueness, we find it useful
to group the terms as
\begin{equation}
\label{eq:spectral-2}
K_t(x,y)= \sum_{\lambda\ge 0} 
e^{-\lambda t} P_\lambda(x,y) \,,
\end{equation}
where every nontrivial eigenvalue in the series appears exactly once, 
and
\begin{equation}
\label{eq:P(x,y)}
P_\lambda(x,y):=\sum_{j: \lambda_j =\lambda}
\phi_j(x)\phi_j(y) 
\end{equation}
is the integral kernel for the projection onto the eigenspace of $\lambda$. Since the eigenfunctions are orthonormal, we have that $\int_M P_\lambda(x,x)\,=\mathrm{tr}\, P_\lambda=m$.

By the same estimates, the truncated series
$\sum_{\lambda'>\lambda} e^{(\lambda-\lambda' )t}P_{\lambda'}(x,y)$
also converges absolutely, uniformly
for $x,y\in M$ and $t\ge \tau$. Therefore we can represent the spectral projection as the limit
\begin{equation}
\label{eq:P-limit}
P_{\lambda}(x,y)
=\lim_{t\to\infty} 
e^{\lambda t}\left(K_t(x,y)-\sum_{\lambda'<\lambda}e^{-\lambda' t}P_{\lambda'}(x,y)\right) \,.
\end{equation}

\begin{proof}[Proof of Theorem~\ref{thm-constant}] 
Under the hypothesis that the heat 
kernel is geodesically decreasing for an infinite set of times $(t_k)_{k\ge 1}$  we want to show that equation~\eqref{eq:constant} holds for
every eigenvalue $\lambda$ of $-\Delta$, that is, 
$P_\lambda(x,x)=m/\Vol(M)$, where
$m$ is the multiplicity of $\lambda$.

As discussed above, for each 
$t>0$ and $x\in M$ the spectral expansion
\[S(t)= K_t(x,x)= \sum_{\lambda\ge 0 }
e^{-\lambda t} P_\lambda(x,x) 
\]
converges absolutely
to an analytic function of $t$.
Standard arguments
guaranteeing uniqueness may be used. This is
enough to
determine each term in the series
(cf. Theorem 70 \cite{Canzani}). 
Here we will only argue that 
$P_\lambda(x,x)$ is
constant on $M$ for each eigenvalue $\lambda$.

If that is not the case, let $\lambda$ be the smallest eigenvalue for which the spectral projection depends on~$x$.
By hypothesis and  Lemma \ref{lem-const}  the function $S(t_k)$ is known to be independent of $x$ for a non-repeating  sequence of 
positive times $(t_k)$ that does converge to zero. If
the sequence is unbounded, then passing to a subsequence (again denoted by $(t_k)$) we may assume that
$\lim t_k=\infty$.
By equation~\eqref{eq:P-limit},
$$
P_{\lambda}(x,x)= \lim_{k\to\to\infty}
e^{\lambda t_k}\left(
K_{t_k}(x,x)-\sum_{\lambda'<\lambda}e^{-\lambda't_k}P_{\lambda'}(x,x)
\right)\,.$$
But by Lemma \ref{lem-const} 
the right hand side does not depend on $x$, a contradiction. We conclude that 
$P_\lambda(x,x)$ is constant on
$M$ for every $\lambda\ge 0$. The value of the constant is determined by integrating over $M$.

If, on the other hand, $(t_k)$ is bounded, we may assume (by passing to a subsequence)
that $\lim t_k=\tau$ for some $\tau>0$. Since $S(t)$ is analytic, 
it is uniquely determined by the values
$S(t_k)$. Namely, its Taylor coefficients about
about $t=\tau$
are determined recursively by
$$
a_{j}\ =\ \lim_{k\to\infty}\ 
\frac{(j)!}{(t_k-\tau)^j}\ 
\left(
S(t_k)-\sum_{0\le i<j} a_i \frac{(t_k-\tau)^i}{i!} \right)
$$
for $j\ge 0$.
Since $S(t_k)$ does not depend on $x$, the same is true for the coefficients $a_j$.
Therefore 
$S(t)=\sum_j a_k(t-\tau)^j$ is 
is constant in $x$ for $t$ in some neighborhood of
$\tau$. By unique continuation,
$S(t)$ is independent of $x$ {\em for all $t>0$. This has strengthened our hypothesis significantly. As a consequence we} are now in a position
to apply equation~\eqref{eq:P-limit},
and conclude as in the first case
that $P_\lambda(x,x)$ is constant on
$M$ for every $\lambda\ge0$.
\end{proof}

For later use, we note that equation~\eqref{eq:constant} implies a similar relation 
for the gradients of the eigenfunctions. 

\begin{lemma} \label{lem-gradient-const}
Let $\phi_1,\dots,\phi_m$ be eigenfunctions
of $-\Delta$  for the same eigenvalue $\lambda$.
If  $\sum_j \phi_j^2$ is constant on $M$, then
$\sum_{j}|\nabla \phi_j|^2$ is constant. 
\end{lemma}

\begin{proof}
For $j=1,\ldots, m$, we have by the chain rule
$$
|\nabla\phi_j|^2 = \mathrm{div}\, (\phi_j\nabla\phi_j) -
\phi_j (\Delta \phi_j) = \left(\tfrac12 \Delta + \lambda\right)\phi_j^2\,,
$$
where we have used that $\phi_j$ is an eigenfunction
for $\lambda$. If $\sum_j\phi_j^2=c^2$ is
constant on $M$, it follows that
$$
\sum_{j=1}^m |\nabla \phi_j|^2 = \left(\tfrac12 \Delta + \lambda\right)
\sum_{j=1}^m \phi_j^2= \lambda c^2\,,
$$
which is constant on $M$, as claimed.
\end{proof}

%%%%%%%%%%%%%%%%%%%%%%%%%%%%%%%%%%%%%%%%%%%%%%%%%%%%%%%%%%%%%%%%%%%%%%
\section{\bf The double eigenvalue case} 
\label{sec-double}
%%%%%%%%%%%%%%%%%%%%%%%%%%%%%%%%%%%%%%%%%%%%%%%%%%%%%%%%%%%%%%%%%%%%%%

Let $(M,g)$ be a compact connected manifold, and $\lambda>0$ an eigenvalue of the Laplace-Beltrami operator $-\Delta_g$. In this section, we consider the case where $\lambda$ has multiplicity two. As in the previous section,
we omit the metric $g$ from the notation whenever there is no
danger of confusion.

Let $\phi_1,\phi_2$ an orthonormal basis for the eigenspace of $\lambda$. By Lemma~\ref{lem-principal-m}, $\phi_1^2+\phi_2^2=c^2$,
where $c$ is the constant from equation~\eqref{eq:def-c} with $m=2$.
For $x\in M$, we write in polar coordinates
\begin{equation}
\label{eq:polar}
\phi_1(x)+i\phi_2(x)= ce^{i\alpha (x)}\,.
\end{equation}
Although the argument
$\alpha(x)$ is determined only up to an integer multiple, we will see that its gradient defines a smooth 
vector field on $M$.
The flow of this vector field will be used to construct
a coordinate system for $M$.
The level sets
\begin{equation}
\label{eq:level-theta}
M_\theta:=\left\{x\in M\ \big\vert\ \phi_1(x)+i\phi_2(x)=ce^{i\theta} \right\}
\end{equation}
will play a key role in the construction of the coordinate system. Note that 
the nodal sets of $\phi_1,\phi_2$
can be written as disjoint unions
$$
\{\phi_1=0\}=M_{\pi/2} \cup M_{-\pi/2}\,,
\qquad 
\{\phi_2=0\}= M_0 \cup M_{-\pi}\,.
$$

We digress now and record a couple of observations that will
ease our later calculations. 
Let $U\subset M$ be an open connected
subset where $\alpha$ is represented by a
smooth function. By the chain rule
$$
\nabla\phi_1+i\nabla \phi_2=ice^{i\alpha}\nabla\alpha\,,
$$
that is, $\nabla\alpha= c^{-2}(\phi_1\nabla\phi_2-\phi_2\nabla\phi_1)$
and $|\nabla\alpha|^2 = c^{-2}(|\nabla\phi_1|^2+|\nabla\phi_2|^2)$
on~$U$.  By Lemma~\ref{lem-gradient-const}, $|\nabla\alpha|=\lambda^{\frac12}>0$.
It follows that  
\begin{equation}
\label{eq:def-V}
V:= \frac{\nabla\alpha}{|\nabla\alpha|}=c^{-2}\lambda^{-\frac12}(\phi_1\nabla\phi_2-\phi_2\nabla\phi_1)\,,
\end{equation}
extends to a global vector field on $M$ that provides a field of unit normals on each of the hypersurfaces $M_\theta$. 
By 
the Implicit Function Theorem,
these hypersurfaces are submanifolds of codimension one; since 
$M$ is connected they are all diffeomorphic to $M_0$. 

Note that $\alpha$ is harmonic and $V$ is divergence-free. 
Indeed, by the product rule
$$
\mathrm{div}\, (\phi_1\nabla\phi_2-\phi_2\nabla\phi_1)
= \phi_1\Delta\phi_2 - \phi_2 \Delta\phi_1 = 0
\,,
$$
where the eigenvalue equation was used in the last step.

Before we proceed any further, let us share a few observations. We will work on coordinate charts on $M$ where the equation makes sense, the vector field is smooth (as a consequence of the eigenfunctions being smooth), and the argument $\alpha(x)$ is well-defined. This allows the use of Picard-Lindel\"of Theorem to show existence and uniqueness for the initial-value problem
\begin{equation}
\label{eq:IVP}
\frac{d}{ds}\eta(s)=V(\eta(s))\,,\qquad \eta(0)=x\,.
\end{equation} 
The flow $(\Psi_s)_{s\in\RR}$ generated by $V$ is defined by the property that for $x\in M$, the
function $\eta(s)=\Psi_s(x)$ 
solves equation \eqref{eq:IVP}.
 The Picard-Lindel\"of theorem guarantees that $\Psi$ depends smoothly on both variables. We will freely use that the flow defines a group 
of diffeomorphisms on $M$, satisfying $\Psi_s\circ\Psi_t=\Psi_{s+t}$
 for all $s,t\in\RR$. 

The following lemma shows that the 
integral curves of $V$ are unit speed geodesics

 \begin{lemma} 
 \label{lem-flow}
Let $\alpha$ and $V$ be as in equations~\eqref{eq:polar}
and~\eqref{eq:def-V}. 
Then the solutions of equation \eqref{eq:IVP} are 
geodesics parametrized by arc length.
 \end{lemma}

 \begin{proof} Let $U\subset M$ be a connected domain where $\alpha$ is well-defined.
It suffices to consider, at any given $x\in U$, solutions $\eta(s)$ with $\eta(0)=x$, for $s\ge 0$ sufficiently small. Since $V$ is a unit vector field, the solutions are parametrized by arc length. The Mean Value Inequality implies
$$
|\alpha(\eta(s))-\alpha(x)|
\ \le \ \bigl(\sup |\nabla \alpha|\bigr)\,
|\eta(s)-x|\ \le \ \lambda^{1/2}s\,,
$$
where the last step uses that $|\nabla\alpha|=\lambda^{1/2}$ is constant, $V$ is a unit vector field, and $\eta$ is parametrized by arc length.  On the other hand,  since $V$ 
is the unit vector field in the direction of $\nabla\alpha$, we have
$$\frac{d}{ds} \alpha(\eta(s))=
\bigl\langle V(\eta(s)),\nabla \alpha(\eta(s))\bigr\rangle
 = |\nabla\alpha|=\lambda^{1/2}\,.
$$
After integrating over $s$, we see that
equality holds in the preceding inequality, and in particular
$|\eta(s)-x| =s$ so long as $\eta(s)$ remains in $U$. This proves that $\eta$ is a minimal geodesic.
\end{proof}

We introduce the following notation.
Given a diffeomorphism $\psi$
on a manifold $N$, and $b>0$, we refer to
\begin{equation}
\label{eq:diffeo}
M_{\psi,b}:= N\times \RR \big\slash
\bigl\{
(x,h)\sim \bigl(\psi(x),h+b\bigr),\ 
x\in N, h\in\RR\bigr\}\,.
\end{equation}
as a {\em mapping torus}.
We are ready to construct coordinates on $M$.

\begin{lemma} 
\label{lem-diffeo}
Let $\lambda$ be a double eigenvalue of $-\Delta_g$
on a compact connected manifold 
$(M,g)$, and let
$\phi_1,\phi_2$ be a pair of orthonormal eigenfunctions for~$\lambda$.

If $\phi_1^2+\phi_2^2=c^2$ is constant on $M$, then, up
to isometry,  $M=M_{\psi, b}$,
where $\psi$ is diffeomorphism 
of a submanifold $N$ of codimension one, $b=2\pi\lambda^{-1/2}$, and the metric $g$ has the form
\begin{equation}\label{eq:gh}
    \langle v,v\rangle_{g}=
\langle \dot x,\dot x\rangle_{g_h}+ \dot h^2
\end{equation}
for any tangent vector $v=(\dot x,\dot h)$ at $(x,h)\in M_{\psi,b}$. 
\end{lemma}

\noindent{\bf Remark.}
The
metrics $(g_h)_{h\in\RR}$ on $N$
are compatible with the equivalence relation in equation~\eqref{eq:diffeo};
in particular $\psi$ is an isometry
from $(N,g_h)$ to $(N,g_{h+b})$
for each $h\in\RR$. Moreover  the 
volume element $\sqrt{\operatorname{det} g_h}$ does not depend on $h$.

\begin{proof}
Let the argument function, the level sets $M_\theta$,
and the unit vector field $V$ be as in
equations~\eqref{eq:polar}-\eqref{eq:def-V}, and
let $(\Psi_s)_{s\in\RR}$ be the flow of $V$. 
Set $N:=M_0$ and $\psi:=\Psi_{-b}\big\vert_{M_0}$.
Choose the argument function $\alpha$ to take  
values in $[0,2\pi)$; it is smooth except for a jump 
of $2\pi$ across $M_0$. 
Since $\alpha(\Psi_s(x))=\lambda^{1/2}s$ modulo $2\pi$, 
we recognize $b$ as the time of first return
and $\psi$ as the (inverse) Poincar\'e map for $M_0$.
Thus $\psi$ defines a diffeomorphism of $N$.

Define $F:N\times\RR\to M$ by
\begin{equation}
\label{eq:def-F}
F(x,h) := \Psi_h(x)\,.
\end{equation}
We claim that $F$ can be factored as the composition of the 
quotient map from $N\times\RR$ to $M_{\psi,b}$ with a 
diffeomorphism $\bar F: M_{\psi,b}\to M$.

$\bar F$ is well-defined, since
$$F(\psi(x),h+b)= \Psi_{h+b}(\psi(x))
= \Psi_{h+b}\circ \Psi_{-b}(x)= \Psi_h(x) = F(x)
$$
for all $x\in M_0$ and $h\in\RR$.
It is injective, since for $x,x'\in M_0$ and $h,h'\in [0, b)$ we have $F(x,h)=F(x',h')$ if and only if $h=h'$ and $\Psi_h(x)=\Psi_h(x')$, which forces $x=x'$ because $\Psi_h$ is invertible. It is also surjective, since for 
any point $x\in M$ taking $h:=\lambda^{-1/2}\alpha(x)$ and 
$y:=\Psi_{-h}(x)\in M_0$ yields $F(y,h)
=\Psi_h\circ \Psi_{-h}(x)=x$. 
By construction, both $\bar F$ and its inverse are smooth.
This proves the first claim.

We pull back
the metric $g$ from $M$ to $M_{\psi,b}$
through the diffeomorphism~$\bar F$. In the same way, we define a metric on the infinite 
cylinder $N\times \RR$ by pulling back $g$ through~$F$.
In a convenient abuse of of notation, we denote the pulled-back 
metrics again by~$g$. In this notation, $F$ is the 
quotient map from $(N\times \RR,g)$ 
to $(M_{\psi,b},g)$. The argument function is 
$\alpha(x,h)=\lambda^{1/2}h$, and
the flow acts by vertical translation, $\Psi_s(x,h)=(x,h+s)$.

Since $V$ is a unit vector field, orthogonal
to the horizontal hypersurfaces $N\times\{h\}= M_{\lambda^{-1/2}h}$,
the metric satisfies equation~\eqref{eq:gh}.
Furthermore, since $V$ is divergence-free, the flow preserves the
volume on $M$, that is, $\sqrt{\operatorname{det} g}(x,h)~=~\sqrt{\operatorname{det} g_h}(x)$ does not 
depend on $h$.
\end{proof}

The lemma implies in particular that
$\ell^2\lambda$ is an eigenvalue of $-\Delta_g$
for each $\ell\ge 1$, with eigenfunctions
$c\, \cos(\ell \lambda^{1/2}h)$
and $c\, \sin(\ell \lambda^{1/2}h)$.
There are also implications for topological invariants, e.g., the 
fundamental group of $M$ contains $\ZZ$ as a factor, and 
its Euler characteristic vanishes.

In the special case of the principal eigenvalue
$\lambda=\lambda_1$, much more can be said:
the mapping torus becomes a product space
(with the product metric),
and $N$ is connected as well as totally geodesic, see Theorem~\ref{thm-product}. 
The crucial ingredient is Lemma~\ref{lem-convex}, which implies 
that the horizontal hypersurfaces
$\{h=h_0\}$ are totally geodesic in $M_{\psi,b}$.

\begin{proof}[Proof of Theorem~\ref{thm-product}]
Assume that the heat kernel on $(M,g)$ is geodesically
decreasing, and the principal eigenvalue $\lambda_1$ of $-\Delta_g$
has multiplicity two. By Lemma~\ref{lem-principal-m},
the eigenfunctions satisfy $\phi_1^2+\phi_2^2=c^2$
for some constant $c>0$.
By Lemma~\ref{lem-diffeo},
the manifold is given by $(M_{\psi,b},g)$, 
as described in equations~\eqref{eq:diffeo}
and~\eqref{eq:gh}. 
We claim that the metrics $g_h$ are locally constant in $h$. 

Given $x_0\in N$ and $h_0\in \RR$, choose coordinates 
$\xi=(\xi_1,\dots,\xi_{n-1})$ for $N$ in a connected 
neighborhood $U$ of $x_0$, and let $I$ be a short open 
interval containing $h_0$. In these coordinates 
$\alpha(\xi,h)=2\pi h/b$
and $V(\xi,h)\equiv(0,1)$.
By Lemma~\ref{lem-convex}, the horizontal hypersurfaces $\{h=h_0\}$ are totally geodesic.
Since a geodesic is uniquely determined 
by its initial position and velocity, any 
geodesic that 
is tangent to the hypersurface $\{h=h_0\}$
initially remains in that hyperplane forever.

We will make use of one component of the variational
equation for the energy functional associated with the length.
The energy of a curve in $U\times I$ parametrized by
$\gamma(s) =(\xi(s),h(s))$
is given by
\[  E(\gamma) 
= \int \langle \dot \xi,\dot \xi\rangle_
{g_{h}(\xi)} + \dot h^2\, ds\,.
\]
For clarity, we have 
suppressed the variable $s$ in the notation.
Minimizers for $E$ among curves with fixed endpoints
are length-minimizing geodesics, and the minimum value
is the squared  distance between the endpoints. Variations in the last coordinate about a minimizing curve, given by $h_{\epsilon} = h +\epsilon \eta$, yield the variational equation
 $$ \int  \eta\, \partial_h 
 \langle \dot\xi,\dot\xi\rangle_{
 g_{h}(\xi)}
 + 2\dot\eta \dot h\,ds \ =\ 0\,,$$
which holds for all smooth (scalar-valued) 
test functions $\eta(s)$ of compact support.

For any geodesic that remains in a hyperplane
$\{h=h_0\}$, the last term of the integrand
vanishes. Since $\eta$ is arbitrary, it follows that 
$\partial_h\langle \dot x, \dot x\rangle_{g_h(x)}$ vanishes identically along the
geodesic.
Since such a geodesic can emanate in 
any horizontal direction $(\dot \xi_0,0)$ 
from any point $(\xi_0,h_0)\in U\times I$,  we conclude by connectedness that the metric $g_h$ does not 
depend on $h$. Thus $g=g_0\otimes 1$ on $N\times \RR$.  By Pythagoras' identity, the distance function on
$M_{\psi,b}$ satisfies
\begin{align*}
\mathrm{dist}_{M_{\psi,b}}\,\bigr((x,h),
(x',h'))^2 &=
\min \Bigl\{
\mathrm{dist}_{N}(x,x')^2+|h-h'|^2,\\
&\qquad \qquad
\mathrm{dist}_{N}(\psi(x),x')^2+|h+b-h'|^2
\Bigr\}
\end{align*}
for $x,x'\in N$ and $0\le h\le h'\le b$.

\smallskip 
Since the metric is compatible with the
equivalence relation in equation~\eqref{eq:diffeo},
the Poincar\'e map $\psi$ is an isometry
of $(N,g_0)$.
It remains to prove that $\psi $ is the identity.
We proceed by contradiction.

Suppose that 
$\psi(x)\ne x$ for some $x\in N$. As in the proof of Proposition~\ref{prop-generic}, we will show that the 
projection of the heat kernel
onto the principal 
eigenspace violates monotonicity along 
the vertical geodesic parametrized by $\beta(s)=(x,sb)$.
The geodesic meets the cut locus of $(x,0)$ at the first 
time $s>0$ where $\beta(s)$ is equidistant to $(x,0)$ and $(\psi(x),b)$.
Since 
$$\mathrm{dist}_{M_{\psi,b}}\,\bigl((x,0),\beta(s)\bigr) \ \le 
\ \frac{b}{2} \ < \ \mathrm{dist}_{M_{\psi,b}}\,\bigl((\psi(x),b),\beta(s)\bigr)\,,
$$
for any $s\in (0,\frac12)$, it follows that the geodesic is minimal up to some $s^*>\frac12$.

Recall that $\phi_1(x,h)=c\cos (2\pi h/b)$,
$\phi_2(x,h)=c\sin (2\pi h/b)$ form an orthonormal basis for the principal eigenspace, where $c$ is given by equation~\eqref{eq:def-c}. 
 The integral kernel of the spectral projection is given by $$P_{\lambda_1}\bigl((\xi,h),(\xi',h')\bigr)
= c^2 \cos(2\pi (h-h')/b)\,.$$
Clearly,
$
P_{\lambda_1}\bigl((\xi, 0),\beta(s)\bigr)=
c^2 \cos(2\pi s)$
increases along the geodesic for $\frac12\le s\le s^*$, 
contradicting  monotonicity.
\end{proof}

We return to the general case of a double
eigenvalue $\lambda>\lambda_1$. Though Lemma~\ref{lem-diffeo} applies, we
have less information about the metric
than for $\lambda=\lambda_1$, and
do not know if $M$ is necessarily
a product
(as in the conclusion of 
Theorem~\ref{thm-product}).
The following result shows that we may
take the base
space $N$ in equation~\eqref{eq:diffeo}
to be connected, at the expense
of enlarging $b$.

\begin{prop} \label{prop-double}
Let $(M,g)$ be a compact connected Riemannian
manifold whose heat kernel is geodesically
decreasing for a non-repeating sequence of positive times
$(t_k)_{k\ge 1}$ that does not converge to zero
(as in the hypothesis of Theorem~\ref{thm-product}). If $\lambda$ is an eigenvalue of $-\Delta_g$ of multiplicity two, then the nodal set of
any eigenfunction  $\phi$ associated with $\lambda$ has an even number of components.

If the number of components is $2k$, let 
$N'$ be one of them, 
and set $b'=2\pi k\lambda^{-\frac12}$.
There exists a diffeomorphism $\psi'$ of $N'$ such that,
up to isometry, $M$ is a mapping torus
$N'_{\psi',b'}$ as in
equation~\eqref{eq:diffeo}, 
and $g$ satisfies equation~\eqref{eq:gh}.
\end{prop}

\noindent {\bf Remark.}\ 
For $k=1$ we recover our previous result in Lemma~\ref{lem-diffeo}.

\begin{proof} By Lemma~\ref{lem-diffeo},
we can identify $M$ with $(M_{\psi,b},g)$,
given by equations~\eqref{eq:diffeo} and~\eqref{eq:gh}, where $N$ is a submanifold of codimension one, 
$\psi$ is a diffeomorphism of $N$, and $b=2\pi\lambda^{-\frac12}$. 
In these coordinates, the eigenfunction is a constant multiple of
$\sin(\lambda^{1/2}(h\!-\!h_0))$;
by making a suitable translation on the infinite cylinder $N\times \RR$
we may assume that $h_0=0$. The nodal set
of $\phi$ is the disjoint union 
$$
\{\phi=0\}= \bigl(N\times\{0\}\bigr)\cup
\bigl(N\times 
\{\tfrac{b}{2}\}\bigr)\,,
$$
with the metrics $g_0$ and $g_{b/2}$, respectively. 
The two pieces are interchanged by the diffeomorphism $\Psi_{b/2}$.
If $N$ has
$k$ components, then the nodal
set has $2k$ components. Performing another translation by $b/2$, if necessary, we may assume that $N'$ is a component of $N\times\{0\}$.

Since $\psi:N\to N$ is a diffeomorphism,
it permutes the components of~$N$.
We claim that the permutation is a single cycle of length $k$. To see this, note that the component
$N'\times \RR$ of $N\times\RR$ is mapped under the quotient map in equation~\eqref{eq:diffeo}
to a component of $M_{\psi,b}$.
Since $M$ is connected, the quotient map 
is surjective.
By the pigeon-hole principle this means that each component of $N$ appears exactly once
among the the iterated images
$\psi^\ell (N')$ for $\ell=0,\ldots, k-1$, and $\psi^kN'=N'$.
Setting $\psi'=\psi^{k}\big\vert_{N'}$
and inspecting the equivalence relation
in equation~\eqref{eq:diffeo}, we
conclude that $M_{\psi,b}=M_{\psi',b'}$, and
that the metrics $g_h$ on $N'$ inherited from
$N\times\RR$
are compatible with the equivalence relation.
\end{proof}

Proposition~\ref{prop-double} implies
that for each $\ell\ge 1$
the function $\phi(\ell h/k)$
is an eigenfunction of $-\Delta_g$, with eigenvalue $(\ell/k)^2\lambda$.  
In particular, $\lambda/k^2$ is a candidate
for the principal eigenvalue 
of $-\Delta_g$ on $M$.

Unfortunately we cannot use the proposition inductively
to further decompose $N'$, because we do not know whether the
heat kernel on $(N',g)$ is geodesically
decreasing. However, there is one case where more 
can be said.

\begin{corr} 
\label{cor-double}
Let $(M,g)$ be a compact connected Riemannian
manifold of dimension two such that the Laplace-Beltrami operator has a double eigenvalue $\lambda$.
If the heat kernel on $M$ is geodesically
decreasing for some non-repeating,
positive sequence of times $(t_k)_{k\ge 1}$ that does not
converge to zero
then, up to isometry, $M$ is a flat torus
$\RR^2/\Lambda$. If
$\lim t_k=\infty$, then the lattice
$\Lambda$ is
rectangular.
\end{corr}

\begin{proof} By Proposition~\ref{prop-double},
$M$ is a mapping torus given by equation~\eqref{eq:diffeo},
where $N$ is a connected manifold
of dimension one (hence, necessarily diffeomorphic to $\SS^1$),
equipped with a metric $g$  of the form \eqref{eq:gh}.
By Lemma~\ref{lem-diffeo},
the volume element $\sqrt{\operatorname{det} g_h}$ does not
depend on $h$. Since $N$ has dimension one, 
we can take $g_h=g_0$ equal to a
uniform metric on $\SS^1$. Thus, the metric $g$
on the infinite cylinder $\SS^1\times \RR$
is flat, and the isometry $\psi$ is
either a translation or a reflection.

If $\psi$ is a reflection then $M$ is a
flat Klein bottle, up to dilation and isometry equal to the manifold $\mathcal{K}$ in equation~\eqref{eq:def-K}. But this
is excluded by the 
remark after Proposition~\ref{prop-K},
Hence $\psi$ is a translation, and $M$ is a flat torus $\RR^2/\Lambda.$

If, moreover, 
$\lim t_k=\infty$, then
it follows from 
Theorem~\ref{thm-flat} that the lattice $\Lambda$ is either rectangular, or regular triangular.
Since the honeycomb torus has no eigenvalues of multiplicity two, %(cf. the remark before Proposition~\ref{prop-generic}),
$M$ is a rectangular torus.
\end{proof}

\noindent {\bf Remark.}\ 
It is an open question  whether 
%it is necessary to assume geodesic monotonicity of the heat kernel for an infinite sequence $(t_k)$,
%or whether perhaps 
positivity for a single $t>0$ suffices 
to reach the conclusion of Corollary~\ref{cor-double}.

\subsection*{Acknowledgments}
The second named author has been  supported by a Fields Ontario Postdoctoral Fellowship financed by NSERC Discovery Grant 311685 and NSERC Grant RGPIN-2018-06487 and later by a Mar\'ia Zambrano Postdoctoral Fellowship.

\end{document}